\documentclass[12pt,leqno,a4paper,final]{amsart}
\usepackage{amsmath, amsthm, verbatim, amsfonts, amssymb, color}
\usepackage{mathrsfs}
\usepackage{dsfont}
\usepackage[a4paper,margin=3cm]{geometry}

\theoremstyle{plain}
\newtheorem{theorem}{Theorem}[section]
\newtheorem{corollary}[theorem]{Corollary}
\newtheorem{lemma}[theorem]{Lemma}

\theoremstyle{definition}
\newtheorem{remark}[theorem]{Remark}

\numberwithin{equation}{section}
\renewcommand\labelenumi{\textup{\alph{enumi})}}
\renewcommand\theenumi\labelenumi

\newcommand\Ee{\mathds{E}}
\newcommand\nat{\mathds{N}}
\newcommand\real{{\mathds{R}}}
\newcommand\I{\mathds{1}}
\newcommand\Pp{\mathds{P}}

\newcommand\dup{\mathrm{d}}
\newcommand\eup{\mathrm{e}}

\newcommand{\di}{n} 
\newcommand{\ca}{a} 
\newcommand{\cb}{b}
\newcommand{\cc}{c}
\newcommand{\cd}{d}

\usepackage[notref,notcite]{showkeys}
\usepackage{xcolor}

\usepackage{marginnote}
\begin{document}
\title[Asymptotic Formulas for Heat Kernels]{\bfseries Exact Asymptotic Formulas for the Heat Kernels
of Space and Time-Fractional Equations}


\author[C.-S.~Deng]{Chang-Song Deng}
\address[C.-S.~Deng]{School of Mathematics and Statistics\\ Wuhan University\\ Wuhan 430072, China}
\email{dengcs@whu.edu.cn}
\thanks{Financial support through the National Natural Science Foundation of China (11401442, 11831015) (for Chang-Song Deng) is gratefully acknowledged.}

\author[R.L.~Schilling]{Ren\'e L.\ Schilling}
\address[R.L.~Schilling]{TU Dresden\\ Fakult\"{a}t Mathematik\\ Institut f\"{u}r Mathematische Stochastik\\ 01062 Dresden, Germany}
\email{rene.schilling@tu-dresden.de}

\begin{abstract}
    This paper aims to study the asymptotic behaviour of the fundamental solutions (heat kernels) of non-local (partial and pseudo differential) equations with fractional operators in time and space. In particular, we obtain exact asymptotic
    formulas for the heat kernels of time-changed Brownian motions and Cauchy processes. As an application, we obtain exact asymptotic formulas for the fundamental solutions to the $n$-dimensional fractional heat equations
    in both time and space
    \begin{gather*}
        \frac{\partial^\beta}{\partial t^\beta}u(t,x)
        = -(-\Delta_x)^\gamma u(t,x), \quad \beta,\gamma\in(0,1).
    \end{gather*}
\end{abstract}

\maketitle

\section{Introduction}\label{intro}

We are interested in the asymptotic behaviour at zero and at infinity of time- and space-fractional evolution equations. The simplest examples of such equations are
\begin{gather}\label{pde}
    \textup{(a)}\quad \frac{\partial u}{\partial t} = -(-\Delta_x)^\beta u
    \qquad\text{and}\qquad
    \textup{(b)}\quad \frac{\partial^\beta u}{\partial t^\beta} = \Delta_x u,
\end{gather}
where $\Delta=\Delta_x = \sum_{i=1}^\di\frac{\partial^2}{\partial x_i^2}$ is the Laplace operator on $\real^\di$, $(-\Delta)^\beta$ is the fractional power of the Laplacian of
order $\beta\in(0,1)$,
\begin{gather*}
    -(-\Delta_x)^\beta u(x)
    =\frac{\beta 4^\beta \Gamma\left(\beta + \frac{\smash{\di}}{2}\right)}{\pi^{\di/2} \Gamma(1-\beta)}
    \lim_{\epsilon\downarrow 0}  \int_{|y|>\epsilon} \frac{u(x+y)-u(x)}{|y|^{\di+2\beta}}\,\dup y,
\end{gather*}
and $\frac{\partial^\beta }{\partial t^\beta}$ is the Caputo derivative of order $\beta\in(0,1)$, i.e.
\begin{gather*}
    \frac{\partial^\beta f(t)}{\partial t^\beta}
    =
    \frac{1}{\Gamma(1-\beta)}\frac{\dup }{\dup t}\int_0^t \frac{f(s)-f(0)}{(t-s)^\beta}\,\dup s.
\end{gather*}
Our standard references for fractional derivatives in time is Samko \emph{et al.}~\cite{Samko}, for the fractional Laplacian in space we use Jacob \cite{Jac} and Kwa\'snicki~\cite{Kwa17}. If $\beta = 1$, \eqref{pde} becomes the classical heat equation whose fundamental solution is the Gauss kernel
\begin{gather}\label{gauss}
    p(t,x,y)
    =\frac{1}{(4\pi t)^{\di/2}}\exp\left[-\frac{|x-y|^2}{4t}\right],\quad t>0,\;x,y\in\real^\di,
\end{gather}
which is also the transition probability density of a Brownian motion $X=(X_t)_{t\geq 0}$ in $\real^\di$. Over the past years there has been considerable interest in space, time and space-and-time fractional equations. The paper \cite{hah-uma} by Hahn and Umarov explains how such equations arise as Fokker--Planck and Kolmogorov equations related to the solutions of SDEs, in a series of papers Luchko and co-authors \cite{gor-luc-yam,luchko,luchko-yama} study related Cauchy problems, see also Hu et al.\ \cite{jac-et-al} for fractional-in-time initial value problems with a pseudo-differential operator in space; Butko \cite{butko} investigates Chernoff-type approximations of the semigroups of such equations. Time-and-space fractional Schr\"{o}dinger equations are discussed by Dubbeldam et al.\ \cite{dubbeldam}. There are various generalizations of such problems, e.g.\ in the direction of fractional stochastic differential equations where a space-time noise term is added on the right-hand side, see e.g.\ Yan and Yin \cite{yan}, or in the direction of semi-fractional derivatives (in time) and semi-stable semigroup generators (in space), see Kern et al.~\cite{kern}.

Both equations in \eqref{pde} have interesting probabilistic interpretations. Denote by $S=(S_t)_{t\geq0}$ a $\beta$-stable subordinator ($0<\beta<1$), i.e.\ a non-decreasing L\'evy process on $[0,\infty)$ with Laplace transform
\begin{gather*}
    \Ee\,\eup^{-rS_t}
    =\eup^{-tr^\beta},
    \quad r>0,\;t\geq0.
\end{gather*}
If $S$ and $X$ are stochastically independent, the time-changed process $X_S = (X_{S_t})_{t\geq 0}$ is a rotationally symmetric $2\beta$-stable L\'evy process. By independence, the transition probability density of $X_{S_t}$ is given by
\begin{gather}\label{su11}
    p^S(t,x,y)
    = \int_0^\infty p(s,x,y)\,\dup_s\Pp(S_t\leq s),
\end{gather}
and Bochner~\cite{Boc49} observed that this is the fundamental solution to the space-fractional
equation (\ref{pde}.a). By $\dup_s$ we denote
the (generalized) derivative w.r.t.\ $s$. This
type of time-change is usually called \emph{subordination} (in the sense of Bochner) and the process $X_S$ is said to be \emph{subordinate} to $X$, cf.\ \cite{SSV}.

If we perform a time-change with the generalized right-continuous inverse of $S$,
\begin{gather*}
    S^{-1}_t
    =\inf\left\{s\geq0\,:\,S_s>t\right\}
    =\sup\left\{s\geq0\,:\,S_s\leq t\right\},
    \quad t\geq 0,
\end{gather*}
we get a stochastic process $X_{S^{-1}} = (X_{S_t^{-1}})_{t\geq 0}$ which is trapped whenever $t\mapsto S_t^{-1}$ is constant. Note that the jumps of $t\mapsto S_t$ correspond to flat pieces of $t\mapsto S_t^{-1}$. These traps slow down the original diffusion process $X$, and in the physics literature $X_{S^{-1}}$ is commonly referred to as \emph{subdiffusion}, see e.g.\ \cite{MWW07, MK00, PSW05} for some applications, \cite{MNX08,Mag10, MS15} for sample path properties and \cite{MS04} for a representation as scaling limit of a continuous time random walk with heavy-tailed waiting times between the steps.

Since the length of the trapping periods are, in general, not exponentially distributed, we cannot expect that $X_{S^{-1}}$ is a Markov process. Nevertheless, the transition probability density of each random variable $X_{S^{-1}_t}$, $t>0$, can be expressed as
\begin{gather}\label{su22}
    p^{S^{-1}}(t,x,y)
    =\int_0^\infty p(s,x,y)\,\dup_s\Pp\left(S^{-1}_t\leq s\right),
\end{gather}
and it is not hard to see, using the Fourier--Laplace transform, that $p^{S^{-1}}(t,x,y)$ is the fundamental solution to the time-fractional heat equation (\ref{pde}.b), see e.g.\ \cite[Theorem 5.1]{MS04} or \cite{BM01, MK00}.

Already in the simple setting \eqref{pde}, the densities $p^S$ and $p^{S^{-1}}$ are often not explicitly known -- a notable exception is $p^S$ for the $\beta = \frac 12$-stable subordinator: In this case $X_S$ is the symmetric Cauchy process and its transition probability density is the Poisson kernel on $\real^\di$,
\begin{gather}\label{poisson}
    p(t,x,y)
    = \frac{c(\di)}{t^\di}\left(1+\frac{|x-y|^2}{t^2}\right)^{-\frac{\di+1}2}
    = \frac{c(\di)t}{\left(t^2+|x-y|^2\right)^{(\di+1)/2}},
\end{gather}
where $c(\di) :=\pi^{-(\di+1)/2}\Gamma\left(\frac{\di+1}{2}\right)$. Therefore, it is important to know the asymptotic behaviour of $p^S$ and $p^{S^{-1}}$ at zero and infinity. For the fundamental solution to (\ref{pde}.a) the asymptotics of $p^S$ at infinity is known to be
\begin{gather}\label{jh4d2}
    p^S(t,x,y)
    \,\sim\,
    \frac{c(\di,\beta)t}{\left(|x-y|^2+t^{1/\beta}\right)^{(\di+2\beta)/2}}
    \quad\text{as $|x-y|^2t^{-1/\beta}\to\infty$},
\end{gather}
where $c(\di,\beta):=\beta 4^{\beta}\pi^{-1-\di/2}\sin(\pi\beta)\Gamma\left(\frac{\di+2\beta}{2}\right)\Gamma(\beta)$. For $d=1$ this formula is due to P\'{o}lya \cite{Pol23} who used Fourier methods, and
the case $d\geq 1$ can be found in Blumenthal and Getoor \cite{BG60}. A beautiful short proof is due to Bendikov \cite{Ben94}. Our approach is similar to Bendikov's and we show that this method also yields the asymptotics at zero. If we combine these methods, we can obtain the asymptotics of the heat kernels of the following heat equations with fractional operators in both time and space:
\begin{equation}\label{frac11}
    \frac{\partial^\beta u}{\partial t^\beta}
    = -(-\Delta_x)^\gamma u,
    \quad \beta,\gamma\in(0,1).
\end{equation}

\section{Results}\label{res}

\subsection{General Results}

Let $\mathcal{L}$ be the infinitesimal generator of a Feller process $X=(X_t)_{t\geq 0}$ which takes values in a locally compact separable
metric space $(M,\rho)$. We assume that $X$ has a transition probability density $p(t,x,y)$; note that $p$ is the fundamental solution to the Kolmogorov backward equation  $\frac{\partial u}{\partial t} = \mathcal{L}_x u$. We may replace in \eqref{pde} the Laplace operator $\Delta$ by the generator $\mathcal{L}$. The resulting equations
\begin{gather}\label{pde2}
    \textup{(a)}\quad \frac{\partial u}{\partial t} = -(\mathcal{-L}_x)^\beta u
    \qquad\text{and}\qquad
    \textup{(b)}\quad \frac{\partial^\beta u}{\partial t^\beta} = \mathcal{L}_x u,
\end{gather}
are the Kolmogorov backward equation (\ref{pde2}.a), resp., the master equation (\ref{pde2}.b) of the time-changed processes $X_S$ and $X_{S^{-1}}$, respectively. As before, $S=(S_t)_{t\geq 0}$ is a $\beta$-stable subordinator, and the fundamental solutions to the problems \eqref{pde2} are still given by the formulas \eqref{su11} and \eqref{su22}, with $p(t,x,y)$ being the probability density of $X_t$.

A deep result by Grigor'yan and Kumagai on two-sided heat kernel estimates \cite[Theorem 4.1]{GK08} shows that under some reasonable conditions, $p(t,x,y)$ satisfies
\begin{gather}\label{twosided}
    \frac{c_1}{t^{\di/\alpha}}F\left(c_2\frac{\rho(x,y)}{t^{1/\alpha}}\right)
\leq
    p(t,x,y)
\leq
    \frac{C_1}{t^{\di/\alpha}}F\left(C_2\frac{\rho(x,y)}{t^{1/\alpha}}\right),
\end{gather}
for suitable constants $\di,\alpha>0$, $c_1, c_2, C_1,C_2>0$, the metric $\rho(x,y)$ on $M$, and  a `profile function' $F$ which is either of exponential type
\begin{gather}\label{form1}
    F(r)
    = \exp\left[-r^{\alpha/(\alpha-1)}\right]
    \quad\text{for some $\alpha\geq 2$},
\intertext{or of polynomial type}\label{form2}
    F(r)
    = (1+r^2)^{-(n+\alpha)/2}\quad\text{for some $\alpha>0$}.
\end{gather}

In order to keep the presentation simple, we assume that $p(t,x,y)$ is of the form
\begin{gather}\label{kernel1}
    p(t,x,y)
    =
    \frac{C_1}{t^{\di/\alpha}} F\left(C_2\frac{\rho(x,y)}{t^{1/\alpha}}\right),
    \quad t>0,\; x,y\in M,
\end{gather}
where $\di,\alpha,C_1,C_2>0$ are constants, and $F:[0,\infty)\to(0,\infty)$ is a non-increasing profile function of exponential or polynomial type. This assumption is, if one has in mind \cite[Conjecture 1.1]{JKLS12} and the concrete examples given there, not artificial. Note that our results hold -- with estimates and explicit constants (given in terms of $c_i$ and $C_i$) rather than asymptotics -- if we use \eqref{twosided} instead of \eqref{kernel1}. We leave these obvious adaptations to the reader. Our results might also be interesting for ultrametric spaces $(X,d)$ where heat kernels are often explicitly known and of the form $p(t,x,y) = t \int_0^{1/d_*(x,y)} N(x,s)\,e^{-ts}\,ds$, see Bendikov \cite{Ben18}; $d_*(x,y)$ is an intrinsic ultrametric which defined on the basis of the underlying ultrametric $d$.

For example, \eqref{gauss} is exponential with $M=\real^\di$, $\rho(x,y)=|x-y|$, $\alpha=2$, $C_1=(4\pi)^{-\di/2}$, $C_2=1/2$, and
$F(r)=\eup^{-r^2}$, while \eqref{poisson} is of polynomial type with $M=\real^\di$, $\rho(x,y)=|x-y|$, $\alpha=1$, $C_1=c(\di)$, $C_2=1$, and $F(r)=(1+r^2)^{-(\di+1)/2}$.

Recently, Chen \emph{et al.}~\cite{CKKW17} have established two-sided heat kernel estimates for $p^{S^{-1}}(t,x,y)$ where $S$ is a (not necessarily stable) subordinator and under the assumption that the original heat kernel $p(t,x,y)$ satisfies two-sided estimates of the form \eqref{twosided}.

Our aim is to investigate the exact asymptotic behaviour of the heat kernels $p^S(t,x,y)$ and $p^{S^{-1}}(t,x,y)$ at zero and at infinity. The setting is as described above, and
throughout this subsection we assume that $S$ is
a $\beta$-stable subordinator.

\begin{theorem}[Asymptotics for subordination]\label{mainsub2}
    Assume that $p(t,x,y)$ is given by \eqref{kernel1} and $(S_t)_{t\geq0}$ is a $\beta$-stable subordinator for some $\beta\in(0,1)$.
    \begin{enumerate}
    \item\label{mainsub2-a}
      If $\int_1^\infty s^{\di+\alpha\beta-1}F(s)\,\dup s <\infty$, then as $\rho(x,y)t^{-1/(\alpha\beta)}\to\infty$,
      \begin{gather*}
          p^S(t,x,y)
          \,\sim\,
          C_1\frac{\alpha\beta}{\Gamma(1-\beta)}\,C_2^{-\di-\alpha\beta} \int_0^\infty s^{\di+\alpha\beta-1}F(s)\,\dup s \cdot \rho(x,y)^{-\di-\alpha\beta}t.
      \end{gather*}

      \item\label{mainsub2-b}
      As $\rho(x,y)t^{-1/(\alpha\beta)}\to 0$,
      \begin{gather*}
          p^S(t,x,y)
          \,\sim\,
          C_1F(0+)\frac{\Gamma\left(\frac{\di}{\alpha\smash{\beta}}\right)}{\beta\Gamma\left(\frac{\di}{\alpha}\right)}\,t^{-\di/(\alpha\beta)}.
      \end{gather*}
    \end{enumerate}
\end{theorem}

\begin{theorem}[Asymptotics for inverse subordination]\label{mainsub}
    Assume that $p(t,x,y)$ is of the form \eqref{kernel1} and $(S^{-1}_t)_{t\geq0}$ is an inverse $\beta$-stable
    subordinator for some $\beta\in(0,1)$.
    \begin{enumerate}
    \item\label{mainsub-a}
    If $\int_1^\infty s^{\di-\alpha-1} F(s)\,\dup s<\infty$, then as $\rho(x,y)^{-\alpha/\beta}t\to\infty$,
    \begin{gather*}
        p^{S^{-1}}(t,x,y)
        \,\sim\,
        \begin{cases}
        \displaystyle
        C_1F(0+)\frac{\Gamma\left(1-\frac{\di}{\alpha}\right)}{\Gamma\left(1-\frac{\beta \di}{\alpha}\right)}\, t^{-\beta \di/\alpha},&\text{if\ } \di<\alpha,
        \\[12pt]
        \displaystyle
        C_1\frac{\beta}{\Gamma(1-\beta)}\,F(0+)\,t^{-\beta}\log\left[\rho(x,y)^{-\alpha/\beta}t\right],&\text{if\ } \di=\alpha,
        \\[12pt]
        \displaystyle
        C_1\frac{C_2^{\alpha-\di}\alpha}{\Gamma(1-\beta)} \int_0^\infty s^{\di-\alpha-1} F(s)\,\dup s\cdot \rho(x,y)^{\alpha-\di}t^{-\beta}, &\text{if\ } \di>\alpha.
        \end{cases}
    \end{gather*}

    \item\label{mainsub-b}
    If $F$ is of polynomial type \eqref{form2}, then as $\rho(x,y)^{-\alpha/\beta}t\to 0$,
    \begin{gather*}
        p^{S^{-1}}(t,x,y)
        \,\sim\,
        C_1\frac{C_2^{-\di-\alpha}}{\beta\Gamma(\beta)} \,\rho(x,y)^{-\di-\alpha}t^\beta.
    \end{gather*}

    \item\label{mainsub-c}
    If $F$ is of exponential type \eqref{form1}, then as $\rho(x,y)^{-\alpha/\beta}t\to 0$,
    \begin{gather*}
        p^{S^{-1}}(t,x,y)
        \,\sim\,
        K_1\rho(x,y)^{-\frac{\di(1-\beta)}{\alpha-\beta}} t^{-\frac{\di(\alpha-1)\beta}{\alpha(\alpha-\beta)}}
        \exp\left[-K_2\rho(x,y)^{\frac{\alpha}{\alpha-\beta}}t^{-\frac{\beta}{\alpha-\beta}}\right]
    \intertext{with the constants}
    \begin{aligned}
        K_1
        &:=
        C_1C_2^{-\frac{\di(1-\beta)}{\alpha-\beta}}\sqrt{\frac{\alpha-1}{(\alpha-\beta)\beta}}
        \;(\alpha-1)^{\frac{\di(\alpha-1)(1-\beta)}{\alpha(\alpha-\beta)}}
        \beta^{\frac{\di(\alpha-1)\beta}{\alpha(\alpha-\beta)}},
        \\
        K_2
        &:=
        C_2^{\frac{\alpha}{\alpha-\beta}}(\alpha-\beta)(\alpha-1)^{-\frac{\alpha-1}{\alpha-\beta}}\beta^{\frac{\beta}{\alpha-\beta}}.
    \end{aligned}
    \end{gather*}
    \end{enumerate}
\end{theorem}

\begin{remark}
    We can state Theorem~\ref{mainsub}.\ref{mainsub-c} in the following way:
    \begin{gather*}
        p^{S^{-1}}(t,x,y)
        \,\sim\,
        K_1t^{-\frac{\beta \di}{\alpha}} A^{-\frac{\di(1-\beta)}{\alpha-\beta}} \exp\left[-K_2A^{\frac{\alpha}{\alpha-\beta}}\right]
        \quad\text{as\ \ } A:=\rho(x,y)t^{-\beta/\alpha}\to\infty.
    \end{gather*}
    This shows that there exist constants $c_i=c_i(\di,\alpha,\beta)$, $i=1,2,3,4$, such that
    \begin{gather*}
        c_1t^{-\frac{\beta \di}{\alpha}} \exp\left[-c_2A^{\frac{\alpha}{\alpha-\beta}}\right]
        \leq
        p^{S^{-1}}(t,x,y)
        \leq
        c_3t^{-\frac{\beta \di}{\alpha}}\exp\left[-c_4A^{\frac{\alpha}{\alpha-\beta}}
        \right]
        \quad\text{for all\ $A\geq 1$}.
    \end{gather*}
    This is in line with the two-sided estimates for $p^{S^{-1}}(t,x,y)$ derived by Chen \emph{et al.} in \cite[Corollary 1.5\,(i)]{CKKW17}.
\end{remark}

\subsection{Time-Changed Brownian Motion}

\begin{corollary}\label{cor1}
    Assume that $p(t,x,y)$ is the Gauss kernel \eqref{gauss} and $(S_t)_{t\geq0}$ is a $\beta$-stable subordinator for some $\beta\in(0,1)$.
    \begin{enumerate}
    \item\label{cor1-a}
      As $|x-y|t^{-1/(2\beta)}\to\infty$,
      \begin{gather*}
          p^S(t,x,y)
          \,\sim\,
          \frac{\beta4^\beta\Gamma\left(\frac{\di}{2}+\beta\right)}{\pi^{\di/2}\Gamma(1-\beta)}\,|x-y|^{-\di-2\beta}t.
      \end{gather*}

      \item\label{cor1-b}
      As $|x-y|t^{-1/(2\beta)}\to 0$,
      \begin{gather*}
          p^S(t,x,y)
          \,\sim\,
          \frac{\Gamma\left(\frac{\di+1}{2}\right) \Gamma\left(\frac{\di}{2\smash{\beta}}\right)}{2\beta\pi^{(\di+1)/2}\Gamma(\di)}\,t^{-\di/(2\beta)}.
      \end{gather*}

    \item\label{cor1-c}
    As $|x-y|^{-2/\beta}t\to\infty$,
    \begin{gather*}
        p^{S^{-1}}(t,x,y)
        \,\sim\,
        \begin{cases}
            \displaystyle
            \frac{1}{2\Gamma\left(1-\frac{\beta}{2}\right)}\,t^{-\beta/2},
            &\text{if\ } \di=1,
        \\[12pt]
            \displaystyle
            \frac{\beta}{4\pi\Gamma(1-\beta)}\,t^{-\beta}\log\left[|x-y|^{-2/\beta}t\right],
            &\text{if\ } \di=2,
        \\[12pt]
            \displaystyle
            \frac{\Gamma\left(\frac{\di}{2}-1\right)}{4\pi^{\di/2}\Gamma(1-\beta)}\,|x-y|^{2-\di}t^{-\beta},&\text{if\ } \di\geq 3.
        \end{cases}
    \end{gather*}

    \item\label{cor1-d}
    As $|x-y|^{-2/\beta}t\to 0$,
    \begin{gather*}
        p^{S^{-1}}(t,x,y)
        \,\sim\,
        K_1|x-y|^{-\frac{\di(1-\beta)}{2-\beta}} t^{-\frac{\di\beta}{2(2-\beta)}}
        \exp\left[-K_2|x-y|^{\frac{2}{2-\beta}}t^{-\frac{\beta}{2-\beta}}\right],
    \intertext{with the constants}
        K_1 := \frac{1}{\sqrt{\beta(2-\beta)}}\,\pi^{-\frac{\di}{2}} 2^{-\frac{\di}{2-\beta}}\beta^{\frac{\di\beta}{2(2-\beta)}}
        \quad\text{and}\quad
        K_2 :=(2-\beta)2^{-\frac{2}{2-\beta}} \beta^{\frac{\beta}{2-\beta}}.
    \end{gather*}
    \end{enumerate}
\end{corollary}

\begin{proof}
    Corollary~\ref{cor1} follows directly from Theorem~\ref{mainsub2} and~\ref{mainsub}, respectively. For Part~\ref{cor1-b} we use  Legendre's doubling formula for the Gamma function
    \begin{gather}\label{dupli}
        \Gamma(z)
        \Gamma\left(z+\tfrac12\right)
        =2^{1-2z}\sqrt{\pi}\,\Gamma(2z),
        \quad z>0,
    \end{gather}
    with $z=\di/2$.
\end{proof}

Note that Euler's reflection formula for the Gamma function, $\Gamma(\beta)\Gamma(1-\beta) = \frac{\pi}{\sin\pi\beta}$, shows that Corollary~\ref{cor1}.\ref{cor1-a} coincides with \eqref{jh4d2}, giving an alternative
proof for \cite[Theorem 2.1]{BG60}.

\subsection{The Time-Changed Cauchy Processe}

\begin{corollary}\label{cor2}
    Assume that $p(t,x,y)$ is the Cauchy kernel \eqref{poisson} and $(S_t)_{t\geq0}$ is a $\beta$-stable subordinator for some $\beta\in(0,1)$.
    \begin{enumerate}
    \item\label{cor2-a}
    As $|x-y|t^{-1/\beta}\to\infty$,
    \begin{gather*}
        p^S(t,x,y)
        \,\sim\,
        \frac{\beta2^{\beta-1}\Gamma\left(\frac{\di+\beta}{2}\right)}{\pi^{\di/2}\Gamma\left(1-\frac\beta2\right)}\,|x-y|^{-\di-\beta}t.
    \end{gather*}

    \item\label{cor2-b}
    As $|x-y|t^{-1/\beta}\to 0$,
    \begin{gather*}
        p^S(t,x,y)
        \,\sim\,
        \frac{\Gamma\left(\frac{\di+1}{2}\right)\Gamma\left(\frac{\di}{\smash{\beta}}\right)}{\beta\pi^{(\di+1)/2}\Gamma(\di)}\,t^{-\di/\beta}.
    \end{gather*}

    \item\label{cor2-c}
    As $|x-y|^{-1/\beta}t\to\infty$,
    \begin{gather*}
        p^{S^{-1}}(t,x,y)
        \,\sim\,
        \begin{cases}
            \displaystyle
            \frac{\beta}{\pi\Gamma(1-\beta)}\,t^{-\beta}\log\left[|x-y|^{-1/\beta}t\right],&\text{if\ } \di=1,
        \\[12pt]
            \displaystyle
            \frac{\Gamma\left(\frac{\di-1}{2}\right)}{2\pi^{(\di+1)/2}\Gamma(1-\beta)}\,|x-y|^{1-\di}t^{-\beta},&\text{if\ } \di\geq 2.
        \end{cases}
    \end{gather*}

    \item\label{cor2-d}
    As $|x-y|^{-1/\beta}t\to 0$,
    \begin{gather*}
        p^{S^{-1}}(t,x,y)
        \,\sim\,
        \frac{\Gamma\left(\frac{\di+1}{2}\right)}{\pi^{(\di+1)/2}\beta\Gamma(\beta)}\,|x-y|^{-\di-1}t^\beta.
    \end{gather*}
    \end{enumerate}
\end{corollary}
\begin{proof}
    All assertions follow from the respective cases in Theorems~\ref{mainsub2} and~\ref{mainsub}. For the proof of~\ref{cor2-a} and~\ref{cor2-c}, we use the well-known integral formula for Euler's Beta function
    \begin{gather*}
        B(r,s)
        = \int_0^\infty z^{r-1}(1+z)^{-r-s}\,\dup z
        = \frac{\Gamma(r)\Gamma(s)}{\Gamma(r+s)},
        \quad r,s>0.
    \end{gather*}
    Part~\ref{cor2-a} also needs \eqref{dupli} with $z= \frac 12(1-\beta)$.

    \medskip
    Alternatively, we can obtain~\ref{cor2-a} and~\ref{cor2-b} from Corollary~\ref{cor1}.\ref{cor1-a},~\ref{cor1-b} if we replace in these formulas $\beta$ by $\beta/2$. This follows from the observation that the Cauchy kernel can be obtained from the Gaussian kernel by subordination with a $\frac 12$-stable subordinator. Since the composition of a $\frac 12$-stable and a $\beta$-stable subordinator has the same probability distribution as a $\frac\beta 2$-stable subordinator,~\ref{cor2-a} and~\ref{cor2-b} are special cases of Corollary~\ref{cor1}.\ref{cor1-a},~\ref{cor1-b}.
\end{proof}

\subsection{Fractional Equations in Both Space and Time}

    We can combine the previous results which deal with space and time fractionality separately to obtain the following simultaneous space-time fractional asymptotics. As far as we are aware of, this has not yet been considered in the literature.
\begin{corollary}\label{cor3}
    Let $\beta,\gamma\in(0,1)$ and denote by $p(t,x,y)$ the fundamental solution to \eqref{frac11}.
    \begin{enumerate}
         \item\label{asswq-a}
         As $|x-y|^{-2\gamma/\beta}t\to\infty$,
    \begin{gather*}
        p(t,x,y)
        \sim
        \begin{cases}
        \displaystyle
        \frac{
        \Gamma\big(\frac{1}{2\gamma}\big)
        \Gamma\big(1-\frac{1}{2\gamma}\big)}
        {2\pi\gamma
        \Gamma\big(1-\frac{\beta}{2\gamma}\big)}
        \, t^{-\beta/(2\gamma)},
        &\text{if $n=1$ and $\gamma\in(\frac 12,1)$},
        \\[\bigskipamount]
        \displaystyle
        \frac{\beta}
        {\pi\Gamma(1-\beta)}
        \,t^{-\beta}
        \log\left[|x-y|^{-1/\beta}t\right],
        &\text{if $n=1$ and $\gamma=\frac 12$},
        \\[\bigskipamount]
        \displaystyle
        \frac{2\gamma\Gamma\big(
        \frac{n-2\gamma}{2}
        \big)}
        {2^{1+2\gamma}\pi^{n/2}\Gamma(1-\beta)\Gamma(1+\gamma)}
        \,|x-y|^{2\gamma-\di}t^{-\beta},
        &\text{if $n>2\gamma$ and $\gamma\in(0,1)$}.
        \end{cases}
    \end{gather*}

    \item\label{asswq-b}
    As $|x-y|^{-2\gamma/\beta}t\to 0$,
    \begin{gather*}
        p(t,x,y)
        \sim
        \frac{\gamma 4^\gamma \Gamma\big(\frac n2 + \gamma\big)}{\pi^{n/2}\Gamma(1-\gamma) \beta\Gamma(\beta)}\,|x-y|^{-n-2\gamma}t^{\beta}.
    \end{gather*}
\end{enumerate}
\end{corollary}

The proof of Corollary~\ref{cor3} will be presented in the next section.

If we use $\gamma=1/2$ in Corollary~\ref{cor3}, we recover Corollary~\ref{cor2}.\ref{cor2-a} and \ref{cor2-b}.

\section{Proof of Theorem~\ref{mainsub2}, Theorem~\ref{mainsub} and Corollary~\ref{cor3}}\label{sec2}

For the proof of our main results we need some preparations. Let $(S_t)_{t\geq0}$ be a $\beta$-stable subordinator, $\beta\in(0,1)$. It is well known that $S_1$ has a density $p_\beta(s)$, $s>0$, with respect to Lebesgue measure; moreover, $p_\beta$ is of class $C^\infty(0,\infty)$, bounded, unimodal (i.e.\ it has a unique maximum point) and it has the following asymptotics at zero and infinity, cf.\ \cite[Theorem 4.7.1 (4.7.13) and Theorem 5.4.1]{UZ99},
\begin{gather*}
    p_\beta(s)
    \,\sim\,
    \begin{cases}
        \displaystyle
        \frac{1}{\sqrt{2\pi(1-\beta)}}\,\beta^{\frac{1}{2(1-\beta)}}s^{-\frac{2-\beta}{2(1-\beta)}}
        \exp\left[-(1-\beta) \left(\beta s^{-1}\right)^{\frac{\beta}{1-\beta}}\right],
        &\text{as\ } s\to 0,
    \\[12pt]
        \displaystyle
        \frac{\beta}{\Gamma(1-\beta)}\,s^{-\beta-1},
        &\text{as\ } s\to\infty.
    \end{cases}
\end{gather*}
This allows us to rewrite $p_\beta(s)$ for $s>0$ in the following way
\begin{align}\label{asymptotic1}
    p_\beta(s)
    &=
    \frac{1}{\sqrt{2\pi(1-\beta)}}\,\beta^{\frac{1}{2(1-\beta)}}s^{-\frac{2-\beta}{2(1-\beta)}}
    \exp\left[-(1-\beta)\left(\beta s^{-1}\right)^{\frac{\beta}{1-\beta}}\right]
    \left(1+\phi_\beta(s)\right),
    \\\label{asymptotic2}
    p_\beta(s)
    &=
    \frac{\beta}{\Gamma(1-\beta)}\,s^{-\beta-1}\left(1+\psi_\beta(s)\right),
\end{align}
where $\phi_\beta,\psi_\beta:(0,\infty)\to(-1,\infty)$ are continuous functions satisfying
\begin{gather*}
    \lim_{s\downarrow0}\phi_\beta(s)=0
    \quad\text{and}\quad
    \lim_{s\to\infty}\psi_\beta(s)=0
\intertext{and}
    \lim_{s\to\infty}s^{-\frac{\beta(2\beta-1)}{2(1-\beta)}} \left(1+\phi_\beta(s)\right)
    =
    \frac{\sqrt{2\pi(1-\beta)}}{\Gamma(1-\beta)}\,\beta^{\frac{1-2\beta}{2(1-\beta)}}.
\end{gather*}
Let us denote by $G_\beta$ the distribution function of $S_1$, i.e.
\begin{gather*}
    G_\beta(x)
    = \Pp(S_1\leq x)
    = \int_0^x p_\beta(s)\,\dup s,
    \quad x\geq 0.
\end{gather*}
Because of the scaling property of a $\beta$-stable subordinator, we have for $s,t>0$,
\begin{align*}
      \Pp\left(S_t^{-1}\leq s\right)
      = \Pp\left(S_s\geq t\right)
      = \Pp\left(s^{1/\beta} S_1\geq t\right)
      &=1-\Pp\left(S_1<s^{-1/\beta} t\right)\\
      &=1-G_\beta\left(s^{-1/\beta}t\right).
\end{align*}
Combining this with \eqref{su22}, we have
\begin{gather}\label{expression1}
      p^{S^{-1}}(t,x,y)
      = -\int_0^\infty p(s,x,y)\,\dup_sG_\beta\!\left(s^{-1/\beta}t\right)
      = \int_0^\infty p\left(t^\beta s^{-\beta},x,y\right) \dup G_\beta(s),
\intertext{and, similarly,}\label{expression2}
      p^{S}(t,x,y)
      = \int_0^\infty p(s,x,y)\,\dup_s G_\beta\!\left(t^{-1/\beta}s\right)
      = \int_0^\infty p\left(t^{1/\beta}s,x,y\right) \dup G_\beta(s).
\end{gather}

\begin{proof}[Proof of Theorem~\ref{mainsub2}]
Set $A:=\rho(x,y)t^{-1/(\alpha\beta)}$; from \eqref{expression2} and the assumption on $p(s,x,y)$ we get
\begin{align*}
    p^{S}(t,x,y)
    &= C_1 t^{-\di/(\alpha\beta)} \int_0^\infty s^{-\di/\alpha} F\bigg(C_2 \frac{\rho(x,y)}{\left(t^{1/\beta}s\right)^{1/\alpha}}\bigg)\,\dup G_\beta(s)\\
    &= C_1t^{-\di/(\alpha\beta)}\int_0^\infty s^{-\di/\alpha}F\left(C_2 As^{-1/\alpha} \right) \dup G_\beta(s).
\end{align*}

\medskip\noindent\ref{mainsub2-a}
Since $F$ is bounded, $\int_1^\infty s^{\di+\alpha\beta-1} F(s)\,\dup s<\infty$ implies that $\int_0^\infty s^{\di+\alpha\beta-1} F(s)\,\dup s<\infty$. From \eqref{asymptotic2} and the definition of $G_\beta$ we see
\begin{align*}
    p^{S}(t,x,y)
    &= \frac{C_1\beta}{\Gamma(1-\beta)} t^{-\di/(\alpha\beta)}\int_0^\infty s^{-\di/\alpha-\beta-1}
        F\left(C_2 As^{-1/\alpha} \right)\big[1+\psi_\beta(s)\big]\,\dup s\\
    &=\frac{C_1\alpha\beta}{\Gamma(1-\beta)} C_2^{-\di-\alpha\beta} \frac{t}{\rho(x,y)^{\di+\alpha\beta}}
        \int_0^\infty s^{\di+\alpha\beta-1}F(s) \left[1+\psi_\beta\left(C_2^\alpha A^\alpha s^{-\alpha}\right) \right] \dup s.
\end{align*}
Since $\psi_\beta$ is bounded on $(0,\infty)$, we may use the dominated convergence theorem to get
\begin{gather*}
    \lim_{A\to\infty}t^{-1} \rho(x,y)^{\di+\alpha\beta} p^{S}(t,x,y)
    = \frac{C_1\alpha\beta}{\Gamma(1-\beta)} C_2^{-\di-\alpha\beta} \int_0^\infty s^{\di+\alpha\beta-1}F(s)\,\dup s.
\end{gather*}

\medskip\noindent\ref{mainsub2-b}
Using the dominated convergence theorem
once again, we see
\begin{align*}
    \lim_{A\to 0}t^{\di/(\alpha\beta)} p^{S}(t,x,y)
    &= C_1\lim_{A\to 0} \int_0^\infty s^{-\di/\alpha}F\left(C_2 As^{-1/\alpha}\right) \dup G_\beta(s)\\
    &= C_1F(0+) \int_0^\infty s^{-\di/\alpha}\,\dup G_\beta(s)\\
    &= C_1F(0+) \Ee S_1^{-\di/\alpha}\\
    &\stackrel{\text{(*)}}{=} C_1F(0+)\frac{\Gamma\left(1+\frac{\di}{\alpha\smash{\beta}}\right)}{\Gamma\left(1+\frac{\di}{\alpha}\right)}
    = C_1F(0+)\frac{\Gamma\left(\frac{\di}{\alpha\smash{\beta}}\right)}{\beta\Gamma\left(\frac{\di}{\alpha}\right)};
\end{align*}
in the equality marked by an asterisk (*) we use Lemma~\ref{fmoment} from the appendix with $\kappa=-\di/\alpha$ and $t=1$. The last equality follows from the functional equation $z\Gamma(z)=\Gamma(1+z)$ for the Gamma function.
\end{proof}

\begin{proof}[Proof of Theorem~\ref{mainsub}]
Define $A:=\rho(x,y)^{-\alpha/\beta}t$. Using \eqref{expression1} we get
\begin{gather}\label{expression}
\begin{aligned}
    p^{S^{-1}}(t,x,y)
    &= C_1 t^{-\beta \di/\alpha}
        \int_0^\infty s^{\beta \di/\alpha}F\bigg(C_2 \frac{\rho(x,y)}{\left(t^\beta s^{-\beta}\right)^{1/\alpha}}\bigg)\,\dup G_\beta(s)\\
    &= C_1t^{-\beta \di/\alpha}\int_0^\infty s^{\beta \di/\alpha}F\bigg(C_2 \left(\frac{s}{A}\right)^{\beta/\alpha}\bigg)\,\dup G_\beta(s).
\end{aligned}
\end{gather}

\medskip\noindent\ref{mainsub-a}
We begin with the asymptotics for $A\to\infty$.
\begin{description}
\item[\normalfont Case 1: $\di<\alpha$]
    If we use in \eqref{expression} the monotone convergence theorem and Lemma~\ref{fmoment} from the appendix, we obtain
    \begin{align*}
        \lim_{A\to\infty} t^{\beta \di/\alpha}p^{S^{-1}}(t,x,y)
        &= C_1\lim_{A\to\infty} \int_0^\infty s^{\beta \di/\alpha}F\bigg(C_2\left(\frac{s}{A}\right)^{\beta/\alpha}\bigg)\,\dup G_\beta(s)\\
        &= C_1F(0+)\int_0^\infty s^{\beta \di/\alpha}\,\dup G_\beta(s)\\
        &= C_1F(0+)\Ee S_1^{\beta \di/\alpha}\\
        &= C_1F(0+)\frac{\Gamma\left(1-\frac{\di}{\alpha}\right)}{\Gamma\left(1-\frac{\beta \di}{\alpha}\right)}.
    \end{align*}

\item[\normalfont Case 2: $\di=\alpha$]
    We know that $\int_1^\infty s^{-1} F(s)\,\dup s<\infty$. Inserting \eqref{asymptotic2} into \eqref{expression} yields
    \begin{align*}
        \frac{t^{\beta}}{\log A}p^{S^{-1}}(t,x,y)
        &= \frac{C_1}{\log A} \int_0^\infty s^{\beta} F\bigg(C_2\left(\frac{s}{A}\right)^{\beta/\alpha}\bigg)\,\dup G_\beta(s)\\
        &=\frac{C_1}{\log A}\big(I_1(A,\alpha,\beta) + I_2(A,\alpha,\beta) + I_3(A,\alpha,\beta)\big)
    \end{align*}
    where
    \begin{align*}
        I_1(A,\alpha,\beta)
        :=&\int_0^1 s^{\beta}F\bigg(C_2\left(\frac{s}{A}\right)^{\beta/\alpha}\bigg) p_\beta(s)\,\dup s\\
        \leq& F(0+)\int_0^1 s^{\beta} p_\beta(s)\,\dup s
        \;\leq\; F(0+),
    \end{align*}
    \begin{align*}
        I_2(A,\alpha,\beta)
        :=& \frac{\beta}{\Gamma(1-\beta)} \int_1^\infty s^{-1} F\bigg(C_2\left(\frac{s}{A}\right)^{\beta/\alpha}\bigg)\,\dup s\\
        =& \frac{\alpha}{\Gamma(1-\beta)} \int_{C_2A^{-\beta/\alpha}}^\infty s^{-1} F(s)\,\dup s,
    \end{align*}
    and
    \begin{align*}
        I_3(A,\alpha,\beta)
        :=& \frac{\beta}{\Gamma(1-\beta)} \int_1^\infty s^{-1} F\bigg(C_2\left(\frac{s}{A}\right)^{\beta/\alpha}\bigg) \psi_\beta(s)\,\dup s\\
        =& \frac{\alpha}{\Gamma(1-\beta)} \int_{C_2A^{-\beta/\alpha}}^\infty s^{-1} F(s)\psi_\beta\left(C_2^{-\alpha/\beta} As^{\alpha/\beta}\right) \dup s.
    \end{align*}
    Now we can use Lemma~\ref{j32c} from the appendix to get
    \begin{align*}
        \lim_{A\to\infty} \frac{t^{\beta}}{\log A}p^{S^{-1}}(t,x,y)
        &= C_1\left(0+\frac{\alpha}{\Gamma(1-\beta)} \cdot \frac\beta\alpha\,F(0+) + \frac{\alpha}{\Gamma(1-\beta)}\cdot 0\right)\\
        &= \frac{C_1\beta}{\Gamma(1-\beta)}\,F(0+).
    \end{align*}

\item[\normalfont Case 3: $\di>\alpha$]
    We know that $\int_1^\infty s^{\di-\alpha-1} F(s)\,\dup s<\infty$. Since $F$ is bounded, this yields $\int_0^\infty s^{\di-\alpha-1}F(s)\,\dup s<\infty$. Using \eqref{expression} and \eqref{asymptotic2} we get
    \begin{align*}
        &\rho(x,y)^{\di-\alpha}t^\beta p^{S^{-1}}(t,x,y)\\
        &\quad= \frac{C_1\beta}{\Gamma(1-\beta)} A^{\beta(1-\di/\alpha)} \int_0^\infty s^{\beta \di/\alpha-1-\beta}
        F\bigg(C_2\left(\frac{s}{A}\right)^{\beta/\alpha}\bigg)\big[1+\psi_\beta(s)\big]\,\dup s\\
        &\quad= \frac{C_1C_2^{\alpha-\di}\alpha}{\Gamma(1-\beta)}
        \int_0^\infty s^{\di-\alpha-1}F(s)\left[1+\psi_\beta\left(C_2^{-\alpha/\beta}As^{\alpha/\beta}\right)\right] \dup s.
    \end{align*}
    Since $\psi_\beta$ is bounded on $(0,\infty)$, the dominated convergence theorem gives
    \begin{gather*}
        \lim_{A\to\infty}\int_0^\infty s^{\di-\alpha-1} F(s)\psi_\beta\left(C_2^{-\alpha/\beta} As^{\alpha/\beta}\right) \dup s
        =0.
    \end{gather*}
    Therefore, we have
    \begin{gather*}
        \lim_{A\to\infty} \rho(x,y)^{\di-\alpha}t^\beta p^{S^{-1}}(t,x,y)
        = C_1\frac{C_2^{\alpha-\di}\alpha}{\Gamma(1-\beta)} \int_0^\infty s^{\di-\alpha-1} F(s)\,\dup s.
    \end{gather*}
\end{description}

\medskip\noindent\ref{mainsub-b}
    Now we consider the limit $A\to 0$ for the profile $F(r)=(1+r^2)^{-(\di+\alpha)/2}$ where $\alpha>0$. Applying in \eqref{expression} the dominated
    convergence theorem and Lemma~\ref{fmoment} from the appendix gives
    \begin{align*}
        &\lim_{A\to 0} \rho(x,y)^{\di+\alpha}t^{-\beta} p^{S^{-1}}(t,x,y)\\
        &\quad= C_1\lim_{A\to 0} A^{-\frac{\beta \di}{\alpha}-\beta}
        \int_0^\infty s^{\beta \di/\alpha}\bigg(1+C_2^2\left(\frac{s}{A}\right)^{2\beta/\alpha}\bigg)^{-(\di+\alpha)/2}\,\dup G_\beta(s)\\
        &\quad= C_1\lim_{A\to 0} \int_0^\infty s^{\beta \di/\alpha}\left(A^{2\beta/\alpha}+C_2^2s^{2\beta/\alpha}\right)^{-(\di+\alpha)/2}\,\dup G_\beta(s)\\
        &\quad= C_1C_2^{-\di-\alpha} \int_0^\infty s^{-\beta}\,\dup G_\beta(s)\\
        &\quad= C_1C_2^{-\di-\alpha}\Ee S_1^{-\beta}
        = C_1C_2^{-\di-\alpha} \frac{\Gamma(2)}{\Gamma(1+\beta)}
        =\frac{C_1C_2^{-\di-\alpha}}{\beta\Gamma(\beta)}.
    \end{align*}

\medskip\noindent\ref{mainsub-c}
    Finally we consider $A\to 0$ for the profile $F(r)=\exp\left[-r^{\alpha/(\alpha-1)}\right]$ with $\alpha\geq 2$. Combining \eqref{expression} and \eqref{asymptotic1} yields
    \begin{align*}
        p^{S^{-1}}(t,x,y)
        &= C_1t^{-\frac{\beta \di}{\alpha}} \int_0^\infty s^{\frac{\beta \di}{\alpha}}
        \exp\left[-C_2^{\frac{\alpha}{\alpha-1}} A^{-\frac{\beta}{\alpha-1}}s^{\frac{\beta}{\alpha-1}}\right]p_\beta(s)\,\dup s\\
        &= \frac{C_1}{\sqrt{2\pi(1-\beta)}} \beta^{\frac{1}{2(1-\beta)}} t^{-\frac{\beta \di}{\alpha}}
        \int_0^\infty s^{\frac{\beta \di}{\alpha} -\frac{2-\beta}{2(1-\beta)}}\times\mbox{}\\
        &\qquad\quad\mbox{}\times\exp\left[-C_2^{\frac{\alpha}{\alpha-1}} A^{-\frac{\beta}{\alpha-1}} s^{\frac{\beta}{\alpha-1}}-(1-\beta)\beta^{\frac{\beta}{1-\beta}} s^{-\frac{\beta}{1-\beta}} \right]\big(1+\phi_\beta(s)\big)\,\dup s.
    \end{align*}
    The claim follows with Lemma~\ref{sjg44k} from the appendix.
\end{proof}

\begin{proof}[Proof of Corollary~\ref{cor3}]
    In abuse of notation we denote by $p_{2\gamma}(t,x,y)=p_{2\gamma}(t,|x-y|)$ the heat kernel of the $n$-dimensional rotationally symmetric
    $2\gamma$-stable L\'{e}vy process $X$. We know that the fundamental solution to \eqref{frac11} can be written as $p(t,x,y)=p_{2\gamma}^{S^{-1}}(t,x,y)$, where $S$ is a $\beta$-stable subordinator which is independent of $X$. On the other hand, it follows from the scaling property that
    \begin{gather*}
        p_{2\gamma}(t,x,y)=t^{-n/(2\gamma)}
        p_{2\gamma}(1,t^{-1/(2\gamma)}|x-y|),
    \end{gather*}
    which yields that $p_{2\gamma}(t,x,y)$ is
    of the form \eqref{kernel1} with
    $M=\real^n$, $\rho(x,y)=|x-y|$, $C_1=C_2=1$, $\alpha=2\gamma$, and
    $F(r)=p_{2\gamma}(1,r)$.

\medskip\noindent\ref{asswq-a}
    Using Corollary~\ref{cor1}.\ref{cor1-b} with
    $t=1$ and $\beta$ replaced by $\gamma$, we find that
    \begin{gather*}
        F(0+)=p_{2\gamma}(1,0+)=
        \frac{\Gamma\big(\frac{\di+1}{2}\big) \Gamma\big(\frac{\di}{2\gamma}\big)}
        {2\gamma\pi^{(\di+1)/2}\Gamma(\di)}.
    \end{gather*}
    Moreover,
    \begin{gather*}
        \Ee|X_1|^{-2\gamma}=
        \int_{\real^n}|x|^{-2\gamma}
        p_{2\gamma}(1,|x|)\,\dup x
        =\frac{2\pi^{n/2}}{\Gamma\left(\frac{n}{2}\right)}
        \int_0^\infty s^{n-2\gamma-1}p_{2\gamma}(1,s)
        \,\dup s,
    \end{gather*}
    which, together with the moment formula for
    stable L\'{e}vy processes in Lemma~\ref{levymoment}
    in the appendix, implies that
    \begin{gather*}
        \int_0^\infty s^{n-2\gamma-1}F(s)\,\dup s
        =\frac{\Gamma\left(\frac{n}{2}\right)}
        {2\pi^{n/2}}\,\Ee|X_1|^{-2\gamma}
        =\frac{\Gamma\left(
        \frac{n-2\gamma}{2}
        \right)}{2^{1+2\gamma}\pi^{n/2}\Gamma(1+\gamma)}.
    \end{gather*}
    It remains to apply Theorem~\ref{mainsub}.\ref{mainsub-a}
    to get the first asymptotic formula.

\medskip\noindent\ref{asswq-b}
    Applying Corollary~\ref{cor1}\,\ref{cor1-a}
    with $t=1$ and $\beta$ replaced by $\gamma$, we have
    \begin{gather*}
        p_{2\gamma}(1,r) \sim \frac{\gamma 4^\gamma \Gamma\left(\frac n2+\gamma\right)}{\pi^{n/2}\Gamma(1-\gamma)} \,r^{-n-2\gamma}\quad \text{as $r\rightarrow\infty$}.
    \end{gather*}
    Let $A:=|x-y|^{-2\gamma/\beta}t$. Then it holds from
    \eqref{expression} and the dominated convergence
    theorem that
    \begin{align*}
        &\lim_{A\rightarrow0}|x-y|^{n+2\gamma} t^{-\beta} p(t,x,y)\\
        &\quad\quad=\lim_{A\rightarrow0}|x-y|^{n+2\gamma} t^{-\beta} p^{S^{-1}}_{2\gamma}(t,x,y)\\
        &\quad\quad=\lim_{A\rightarrow0}|x-y|^{n+2\gamma}
        t^{-\beta-\beta n/(2\gamma)}\int_0^\infty
        s^{\beta n/(2\gamma)}p_{2\gamma}\left(1,
        \left(\frac{s}{A}\right)^{\beta/(2\gamma)}
        \right)\,\dup G_\beta(s)\\
        &\quad\quad=\lim_{A\rightarrow0}\int_0^\infty
        s^{-\beta}
        \left(\frac{s}{A}\right)^{(n+2\gamma)\beta/(2\gamma)}
        p_{2\gamma}\left(1,
        \left(\frac{s}{A}\right)^{\beta/(2\gamma)}
        \right)\,\dup G_\beta(s)\\
        &\quad\quad=\frac{\gamma 4^\gamma \Gamma\left(\frac n2+\gamma\right)}{\pi^{n/2}\Gamma(1-\gamma)}
        \int_0^\infty
        s^{-\beta}\,\dup G_\beta(s)\\
        &\quad\quad=\frac{\gamma 4^\gamma \Gamma\left(\frac n2+\gamma\right)}{\pi^{n/2}\Gamma(1-\gamma)}
        \,\Ee S_1^{-\beta}.
    \end{align*}
    Combining this with Lemma~\ref{fmoment} in the appendix,
    we obtain the second formula.
\end{proof}

\section{Appendix}\label{app}

We will need a moment formula for stable subordinators which can be found in Sato~\cite[Eq.\ (25.5), p.\ 162]{Sat99} (without proof but references to the literature). The following short and straightforward derivation seems to be new.

\begin{lemma}\label{fmoment}
    The moments of order $\kappa\in(-\infty,\beta)$ of a $\beta$-stable subordinator $(S_t)_{t\geq 0}$ exist and are given by
    \begin{gather*}
        \Ee S_t^\kappa
        =
        \frac{\Gamma\left(1-\frac\kappa{\smash{\beta}}\right)}
        {\Gamma(1-\kappa)}\,t^{\kappa/\beta},
        \quad t>0.
    \end{gather*}
\end{lemma}
\begin{proof}
    Since $S_t$ has the same probability distribution as $t^{1/\beta}S_1$, it is enough to consider $t=1$. Recall that the Laplace transform of $S_1$ is $\Ee\,\eup^{- t S_1} = \eup^{-t^\beta}$, $t>0$. Substituting $\lambda=S_1$ in the well-known formula \cite[p.~vii]{SSV}
    \begin{gather*}
        \lambda^{-r} = \frac 1{\Gamma(r)}\int_0^\infty \eup^{-\lambda x} x^{r-1}\,\dup x,\quad \lambda > 0,\; r>0,
    \end{gather*}
    and taking expectations yields, because of Tonelli's theorem,
    \begin{gather*}
        \Ee S_1^{-r}
        = \frac 1{\Gamma(r)}\int_0^\infty \Ee\,\eup^{-x S_1} x^{r-1}\,\dup x
        = \frac 1{\Gamma(r)}\int_0^\infty \eup^{-x^\beta} x^{r}\,\frac{\dup x}{x}.
    \end{gather*}
    Now we change variables according to $y=x^\beta$, and get
    \begin{gather*}
        \Ee S_1^{-r}
        = \frac 1{\Gamma(r)} \frac 1\beta\int_0^\infty \eup^{-y} y^{\frac r\beta}\,\frac{\dup y}{y}
        = \frac{1}{r\Gamma(r)} \cdot \frac r\beta\,\Gamma\left(\frac r{\smash{\beta}}\right)
        = \frac{\Gamma\left(1+\frac{r}{\smash{\beta}}\right)}{\Gamma(1+r)}.
    \end{gather*}
    Setting $\kappa = -r$ proves the assertion for $\kappa\in (-\infty,0)$. Note that this formula extends (analytically) to $-r = \kappa < \beta$. Alternatively, use the very same calculation and the formula \cite[p.~vii]{SSV}
    \begin{equation}\label{hg43df}
        \lambda^{r} = \frac r{\Gamma(1-r)}\int_0^\infty \left(1-\eup^{-\lambda x}\right) x^{-r-1}\,\dup x,\quad \lambda > 0,\; r\in (0,1),
    \end{equation}
    to get the assertion for $\kappa\in (0,\beta)$.
\end{proof}

    The following theorem is known in the literature in dimension $n=1$, see \cite[p.~163]{Sat99}. The multivariate setting and the short proof via subordination are new.

\begin{lemma}\label{levymoment}
    Let $(X_t)_{t\geq 0}$ be a rotationally symmetric $\alpha$-stable L\'{e}vy process on $\real^n$ with $0<\alpha<2$. For any $\kappa\in(-n,\alpha)$,
    \begin{gather*}
        \Ee|X_t|^\kappa
        =
        \frac{2^\kappa
        \Gamma\left(\frac{n+\kappa}{2}\right)
        \Gamma\left(1-\frac{\kappa}{\alpha}\right)}
        {\Gamma\left(\frac{n}{2}\right)
        \Gamma\left(1-\frac{\kappa}{2}\right)
        }
        \,t^{\kappa/\alpha},
        \quad t>0.
    \end{gather*}
    If $\kappa \leq -n$ or $\kappa\geq\alpha$ the moments are infinite.
\end{lemma}

\begin{proof}
    Let $(B_t)_{t\geq0}$ be a Brownian motion on $\real^n$ (starting from zero) with transition probability
    density given by \eqref{gauss}, and $(S_t)_{t\geq0}$ be an independent $\alpha/2$-stable
    subordinator, that is an increasing L\'evy process. From Bochner's subordination is well known that the
    time-changed process $B_{S_t}$, $t\geq0$, is a rotationally symmetric $\alpha$-stable L\'{e}vy process on $\real^n$.

    For any $\kappa>-n$ and $t>0$, we have
    \begin{align*}
        \Ee|B_t|^\kappa
        &= \frac{1}{(4\pi t)^{n/2}}\int_{\real^n} |x|^\kappa\exp\left[-\frac{|x|^2}{4t}\right] \dup x\\
        &=\frac{2^{1-n}t^{-n/2}}{\Gamma\left(\frac{n}{2}\right)}\int_0^\infty r^{n+\kappa-1}\exp\left[-\frac{r^2}{4t}\right] \dup r
        =\frac{2^\kappa\Gamma\left(\frac{n+\kappa}{2}\right)}{\Gamma\left(\frac{n}{2}\right)}\,t^{\kappa/2}.
    \end{align*}
    Let $\Ee^B$ and $\Ee^S$ denote the expectations w.r.t.\ $(B_t)_{t\geq0}$ and $(S_t)_{t\geq0}$, respectively. Using Lemma~\ref{fmoment},
    we obtain that for any $\kappa\in(-n,\alpha)$ and $t>0$,
    \begin{gather*}
        \Ee|B_{S_t}|^\kappa
        = \Ee^S\left[\Ee^B|B_{S_t}|^\kappa\right]
        = \frac{2^\kappa\Gamma\left(\frac{n+\kappa}{2}\right)}{\Gamma\left(\frac{n}{2}\right)} \, \Ee S_t^{\kappa/2}
        = \frac{2^\kappa\Gamma\left(\frac{n+\kappa}{2}\right)\Gamma\left(1-\frac{\kappa}{\alpha}\right)}%
        {\Gamma\left(\frac{n}{2}\right)\Gamma\left(1-\frac{\kappa}{2}\right)}\,t^{\kappa/\alpha}.
    \end{gather*}

If $\kappa\leq-n$ or $\kappa\geq\alpha$, we have $\Ee|X_t|^\kappa=\infty$, see \cite[Theorem 3.1.e) and Remark 3.2.d)]{deng-schi15}.
\end{proof}

\begin{remark}
    We want to sketch another, slightly more general proof of Lemma~\ref{levymoment} which avoids the subordination argument. Combining the well-known formulas
    \begin{align*}
        |x|^\kappa
        &= \frac{\kappa2^{\kappa-1}\Gamma\left(\frac{n+\kappa}{2}\right)}{\pi^{n/2}\Gamma\left(1-\frac{\kappa}{2}\right)}
        \int_{\real^n\setminus\{0\}} \big(1-\cos(x\cdot\xi)\big)|\xi|^{-\kappa-n}\,\dup\xi,\quad \kappa\in (0,2),\\
        |x|^\kappa
        &=
        \frac{2^{\kappa}\Gamma\left(\frac{n+\kappa}{2}\right)}{\pi^{n/2}\Gamma\left(-\frac{\kappa}{2}\right)}
        \int_{\real^n\setminus\{0\}}|\xi|^{-n-\kappa}
        \eup^{-\mathrm{i}x\cdot\xi}\,\dup\xi,\quad \kappa\in(-n,0)
    \end{align*}
    (the second formula is to be understood in the sense of L.\ Schwartz distributions) with an Abel-type convergence factor argument and  Fubini's theorem, also yields the moment formula of Lemma~\ref{levymoment}.
\end{remark}

\begin{lemma}\label{j32c}
    Let $\psi:[1,\infty)\to\real$ be a bounded function such that $\lim_{s\to\infty}\psi(s) = 0$ and $\omega:[0,\infty)\to(0,\infty)$ a non-increasing function satisfying $\int_1^\infty s^{-1}\omega(s)\,\dup s<\infty$. For any $c>0$ and $\delta>0$ one has
    \begin{gather*}
        \lim_{A\to\infty}\frac{1}{\log A}\int_{cA^{-\delta}}^\infty s^{-1}\omega(s)\,\dup s
        = \delta\omega(0+)
    \intertext{and}
        \lim_{A\to\infty}\frac{1}{\log A} \int_{cA^{-\delta}}^\infty s^{-1}\omega(s)\psi\!\left(c^{-1/\delta}As^{1/\delta}\right) \dup s
        = 0.
    \end{gather*}
\end{lemma}
\begin{proof}
    The first claim follows easily from l'Hospital's rule. For $n\in\nat$, we can use the first part of the lemma and get
    \begin{align*}
        &\left|\frac{1}{\log A} \int_{cA^{-\delta}}^\infty s^{-1}\omega(s)\psi\!\left(c^{-1/\delta}As^{1/\delta}\right) \dup s\right|\\
        &\qquad\leq\frac{1}{\log A} \int_{cA^{-\delta}}^{\infty} s^{-1}\omega(s)\left|\psi\!\left(c^{-1/\delta}As^{1/\delta}\right)\right|\,\dup s\\
        &\qquad\leq\|\psi\|_\infty \frac{1}{\log A}\left(\int_{cA^{-\delta}}^\infty-\int_{ncA^{-\delta}}^\infty\right)s^{-1}\omega(s)\,\dup s\\
        &\qquad\qquad\mbox{}+\|\psi\I_{[n^{1/\delta},\infty)}\|_\infty\frac{1}{\log A}\int_{ncA^{-\delta}}^\infty s^{-1}\omega(s)\,\dup s\\
        &\qquad\xrightarrow[A\to\infty]{} \|\psi\|_\infty\big(\delta\omega(0+)-\delta\omega(0+)\big)
        + \|\psi\I_{[n^{1/\delta},\infty)}\|_\infty\delta\omega(0+)\\
        &\qquad=\|\psi\I_{[n^{1/\delta},\infty)}\|_\infty\delta\omega(0+)
        \xrightarrow[n\to\infty]{}0,
    \end{align*}
    and this completes the proof.
\end{proof}

The following asymptotic formula for integrals can be
proved by the Laplace method,
see e.g.\ de Bruijn~\cite[Section 4.2, pp.\ 63--65]{Bru58}
for $r_0=0$. 
\begin{lemma}\label{laplacemethod}
    Assume that $-\infty\leq v<w\leq\infty$, $h\in C^2(v,w)$, and $\int_v^w\eup^{-h(r)}\,\dup r<\infty$. Let $r_0\in(v,w)$. If $h(r_0)\geq0$, $h''(r_0)>0$,
    and $h$ is strictly decreasing on $(v,r_0]$ and strictly increasing on $[r_0,w)$, then
    \begin{gather*}
        \int_v^w\eup^{-Ch(r)}\,\dup r
        \,\sim\,
        \eup^{-Ch(r_0)} \sqrt{\frac{2\pi}{Ch''(r_0)}}
        \quad
        \text{as $C\to\infty$}.
    \end{gather*}
\end{lemma}

\begin{lemma}\label{sjg44k}
    Let $\phi:(0,\infty)\to(-1,\infty)$ be a continuous function such that $\phi(0+)=0$
    and $\limsup_{s\to\infty}s^{-\theta}(1+\phi(s))<\infty$ for some $\theta\in\real$. For all constants $\ca\in\real$ and $\cb,\cc,\cd>0$ the following asymptotics holds
    \begin{gather*}
        \int_0^\infty s^{\ca} \eup^{-Bs^{\cb}-\cc s^{-\cd}}\big(1+\phi(s)\big)\,\dup s
        \;\,\sim\;\,
        I(B)
        \quad\text{as $B\to\infty$}
    \end{gather*}
    where the value $I(B)$ is given by
    \begin{align*}
        I(B)
        :=\sqrt{\frac{2\pi}{\cb+\cd}}\,(\cb B)^{-\frac{2(\ca+1)+\cd}{2(\cb+\cd)}}(\cc\cd)^{\frac{2(\ca+1)-\cb}{2(\cb+\cd)}}
        \exp\left[-(\cb+\cd)\left(\cb^{-1}\cc\right)^{\frac{\cb}{\cb+\cd}}\left(\cd^{-1}B\right)^{\frac{\cd}{\cb+\cd}}\right].
    \end{align*}
\end{lemma}

\begin{proof}
    First, we prove that
    \begin{gather}\label{jds4j}
        \int_0^\infty s^{\ca}\eup^{-Bs^{\cb}-\cc s^{-\cd}}\,\dup s
        \;\sim\;
        I(B)
        \quad\text{as $B\to\infty$}.
    \end{gather}
    If $\ca\neq-1$, changing variables according to
    $r=\left(\cc^{-1}B\right)^{(\ca+1)/(\cb+\cd)}s^{\ca+1}$
    gives
    \begin{align*}
        \int_0^\infty s^{\ca}\eup^{-Bs^{\cb}-\cc s^{-\cd}}\,\dup s
        =\frac{1}{|\ca+1|}\left(\cc B^{-1}\right)^{\frac{\ca+1}{\cb+\cd}}
        \int_0^\infty\exp\left[-\cc^{\frac{\cb}{\cb+\cd}}B^{\frac{\cd}{\cb+\cd}}\left(r^{\frac{\cb}{\ca+1}}+r^{-\frac{\cd}{\ca+1}}\right)\right] \dup r.
    \end{align*}
    Lemma~\ref{laplacemethod} with $v=0$, $w=\infty$, $h(r)=r^{\cb/(\ca+1)}+r^{-\cd/(\ca+1)}$, $r_0=(\cb^{-1}\cd)^{(\ca+1)/(\cb+\cd)}$ and $C=\cc^{\cb/(\cb+\cd)}B^{\cd/(\cb+\cd)}$ yields \eqref{jds4j}.

    If $\ca=-1$, we change variables according to $r=\log s+(\cb+\cd)^{-1}\log\left(\cc^{-1}B\right)$, and use Lemma~\ref{laplacemethod} with $v=-\infty$, $w=\infty$, $h(r)=\eup^{\cb r}+\eup^{-\cd r}$, $r_0=(\cb+\cd)^{-1}\log(\cb^{-1}\cd)$ and $C=\cc^{\cb/(\cb+\cd)}B^{\cd/(\cb+\cd)}$ to obtain \eqref{jds4j}.

    We still have to check that
    \begin{gather*}
        \lim_{B\to\infty} B^{\frac{2(\ca+1)+\cd}{2(\cb+\cd)}}
        \exp\left[(\cb+\cd)\left(\cb^{-1}\cc\right)^{\frac{\cb}{\cb+\cd}} \left(\cd^{-1} B\right)^{\frac{\cd}{\cb+\cd}}\right]
        \int_0^\infty s^{\ca}\eup^{-Bs^{\cb}-\cc s^{-\cd}}\phi(s)\,\dup s
        =0.
    \end{gather*}
    To this end, we fix $n\in\nat$ and observe that
    \begin{align*}
        I_1(n,B) := \int_0^{1/n} s^{\ca} \eup^{-Bs^{\cb}-\cc s^{-\cd}} |\phi(s)|\,\dup s
        &\leq\left\|\phi\I_{(0,1/n)}\right\|_\infty \int_0^{1/n} s^{\ca} \eup^{-Bs^{\cb}-\cc s^{-\cd}}\,\dup s\\
        &\leq\left\|\phi\I_{(0,1/n)}\right\|_\infty \int_0^\infty s^{\ca} \eup^{-Bs^{\cb}-\cc s^{-\cd}}\,\dup s.
    \end{align*}
    Moreover, set
    \begin{gather*}
        I_2(n,B)
        := \int_{1/n}^\infty s^{\ca}\eup^{-Bs^{\cb}-\cc s^{-\cd}}|\phi(s)|\,\dup s.
    \end{gather*}
    By our assumption, there exists a constant $C(n)>0$ depending on $n$ such that $1+\phi(s)\leq C(n)s^\theta$ for all $s\geq1/n$. Thus,
    \begin{gather*}
        |\phi(s)|
        \leq 1+(1+\phi(s))
        \leq (ns)^{\theta\vee 0} + C(n)s^\theta
        \leq C(n,\theta) s^{\theta\vee 0},\quad s\geq 1/n,
    \end{gather*}
    where $C(n,\theta):=n^{\theta\vee0}
    +C(n)n^{(-\theta)\vee0}$. Using the
    dominated convergence theorem we deduce
    \begin{align*}
        &B^{\frac{2(\ca+1)+\cd}{2(\cb+\cd)}}
        \exp\left[(\cb+\cd)\left(\cb^{-1}\cc\right)^{\frac{\cb}{\cb+\cd}}\left(\cd^{-1} B\right)^{\frac{\cd}{\cb+\cd}}\right]I_2(n,B)\\
        &\leq C(n,\theta) \int_{1/n}^\infty B^{\frac{2(\ca+1)+\cd}{2(\cb+\cd)}} s^{\ca+(\theta\vee0)}
        \exp\left[(\cb+\cd)\left(\cb^{-1}\cc\right)^{\frac{\cb}{\cb+\cd}} \left(\cd^{-1}B\right)^{\frac{\cd}{\cb+\cd}}-Bs^{\cb}-\cc s^{-\cd}\right]
        \dup s\\
        &\xrightarrow[B\to\infty]{}C(n,\theta)\cdot 0
        = 0.
    \end{align*}
    Combining these calculations gives
    \begin{align*}
        &B^{\frac{2(\ca+1)+\cd}{2(\cb+\cd)}}
        \exp\left[(\cb+\cd)\left(\cb^{-1}\cc\right)^{\frac{\cb}{\cb+\cd}} \left(\cd^{-1}B\right)^{\frac{\cd}{\cb+\cd}}\right]
        \cdot\left| \int_0^\infty s^{\ca} \eup^{-Bs^{\cb}-\cc s^{-\cd}} \phi(s)\,\dup s\right|\\
        &\qquad\leq B^{\frac{2(\ca+1)+\cd}{2(\cb+\cd)}}
        \exp\left[(\cb+\cd)\left(\cb^{-1}\cc\right)^{\frac{\cb}{\cb+\cd}} \left(\cd^{-1}B\right)^{\frac{\cd}{\cb+\cd}}\right]
        \cdot\big(I_1(n,B)+I_2(n,B)\big)\\
        &\qquad\leq B^{\frac{2(\ca+1)+\cd}{2(\cb+\cd)}}
        \exp\left[(\cb+\cd)\left(\cb^{-1}\cc\right)^{\frac{\cb}{\cb+\cd}} \left(\cd^{-1}B\right)^{\frac{\cd}{\cb+\cd}}\right]\\
        &\qquad\qquad\mbox{}\times\left(\left\|\phi\I_{(0,1/n)}\right\|_\infty\int_0^\infty s^{\ca}\eup^{-Bs^{\cb}-\cc s^{-\cd}}\,\dup s+I_2(n,B)\right)\\
        &\qquad\xrightarrow[B\to\infty]{}
        \left\|\phi\I_{(0,1/n)}\right\|_\infty \sqrt{\frac{2\pi}{\cb+\cd}} \,\cb^{-\frac{2(\ca+1)+\cd}{2(\cb+\cd)n}}(\cc\cd)^{\frac{2(\ca+1)-\cb}{2(\cb+\cd)}}+0
        \xrightarrow[n\to\infty]{}0+0
        = 0.
    \end{align*}
    This completes the proof.
\end{proof}

\end{document}